\documentclass{amsart}

\usepackage{amssymb, graphicx}

\usepackage{amsfonts}
\usepackage{latexsym}
\usepackage{amscd}
\usepackage{verbatim}
\usepackage{multirow}

\usepackage{cite}

\usepackage{mathrsfs}
\usepackage{color,ushort,bbm}
\usepackage{enumerate}
\usepackage[all]{xy}

\allowdisplaybreaks[1]

\newtheorem{theorem}{Theorem}[section]
\newtheorem{proposition}[theorem]{Proposition}
\newtheorem{lemma}[theorem]{Lemma}
\newtheorem{corollary}[theorem]{Corollary}
\theoremstyle{remark}
\newtheorem{question}[theorem]{Question}
\newtheorem{remark}[theorem]{Remark}
            
\newtheorem{example}[theorem]{Example}

\newcommand{\CC}{\mathbb{C}}

\newcommand{\tnz}{\otimes}

\newcommand{\im}{\mathrm{im}}
\newcommand{\tr}{\mathrm{tr}}

\newcommand{\ax}{\langle\ushort X\rangle}

\newcommand{\ua}{\ushort a}

\def\ben{\begin{enumerate} }
\def\beq{\begin{equation} }
\def\een{\end{enumerate} }
\def\eeq{\end{equation} }

 \DeclareMathOperator{\disc}{disc}

\newcommand{\cD}{\mathcal D}

\newcommand{\N}{\mathbb{N}}
\newcommand{\Q}{\mathbb{Q}}
\newcommand{\Z}{\mathbb{Z}}

\newcommand{\ul}{\underline}

\begin{document}

\title[On the image of a noncommutative polynomial]
{On the image of a noncommutative polynomial}
\author{\v Spela \v Spenko}
\address{Institute of  Mathematics, Physics, and Mechanics,  Ljubljana, Slovenia} \email{spela.spenko@imfm.si}

\begin{abstract}
Let  $F$ be an algebraically closed field of characteristic zero. We consider the question which subsets of $M_n(F)$ can be images of noncommutative polynomials.
We prove that  a noncommutative polynomial $f$ has only finitely many similarity orbits modulo nonzero scalar multiplication in its image if and only if $f$ is power-central. The union of the zero matrix and a standard open set closed under conjugation by $GL_n(F)$ and nonzero scalar multiplication is shown to be  the image of a noncommutative polynomial.  We investigate the density of the images with respect to the Zariski topology. 
 We also answer  Lvov's conjecture for multilinear Lie polynomials of degree at most $4$ affirmatively.
\end{abstract}

\thanks{2010 {\em Math. Subj. Class.} Primary: 16R30; Secondary: 16K20, 16R99,16S50. }
\maketitle

\section{Introduction}

Let $n$ be a (fixed) integer $\ge 2$, and let $F$ be a field.  We will be concerned with the following problem: 

\smallskip
\centerline{{\em Which subsets of $M_n(F)$ are images of noncommutative polynomials?} }

\smallskip
\noindent
According to \cite{Bel}, this question was ``reputedly raised by Kaplansky". By the image of a  (noncommutative) polynomial $f=f(x_1,\ldots,x_d)$ we mean, of course, the set
 $\im(f) = \{f(a_1,\ldots,a_d)\,|\, a_1,\ldots,a_d\in M_n(F)\}$. An obvious necessary condition for a subset $S$ of $M_n(F)$ to be equal to $\im (f)$ for some $f$ is that $S$ is closed under conjugation  by invertible matrices, i.e., $tSt^{-1}\subseteq S$ for every invertible $t\in M_n(F)$. Chuang \cite{Ch} proved that if $F$ is a finite field, $0\in S$, and if we consider only polynomials with zero constant term, then this condition is also sufficient. This is not true for infinite fields. Say, the set of all square zero matrices cannot be the image of a polynomial \cite[Example, p. 294]{Ch}.

We will consider the case where 
 $F$ is an algebraically closed field of characteristic 0. 
From now on let $M_n$ stand for $M_n(F)$, $M_n^0$ for the space of all trace zero matrices in $M_n$, and $GL_n$ for the group of all invertible matrices in $M_n$.

If $f$ is a polynomial identity, then $\im (f) =\{0\}$. Another important situation where $\im (f)$ is  ``small" is when $f$ is a central polynomial; then $\im (f)$ consists of scalar matrices. What are other possible small images? When considering this question, one has to take into account that if $a\in \im (f)$, then the similarity orbit of $a$ is also contained in $\im (f)$.  The images of many polynomials (for example the homogeneous ones) are also closed under scalar multiplication. Accordingly, let us denote $a^\sim = \{\lambda tat^{-1}\,|\, t\in { GL_n},\lambda\in F\}$.
Is it possible that $\im (f)\subseteq a^\sim$ for some nonscalar matrix $a$? In the $n=2$ case  the answer comes easily: $\im (x_1x_2-x_2x_1)^3 = 
  \left (\begin{array}{cc} 1 & 0 \\ 0 & -1 \end{array} \right)^{\sim}$. One can check this by an easy computation, but the concept behind this example is that the polynomial $(x_1x_2-x_2x_1)^2$ is central, so that 
$\im (x_1x_2-x_2x_1)^3$ can consist only of those trace zero matrices whose determinant is nonzero. Let us also mention that $x_1x_2 - x_2x_1$ also has a relatively small image, namely $M_2^0=\left( \begin{array}{cc} 1 & 0 \\ 0 & -1 \end{array}\right)^\sim\bigcup \left( \begin{array}{cc} 0 & 1 \\ 0 & 0\end{array}\right)^\sim$.
Return now to an arbitrary $n$, and let us make the following definition: A polynomial $f$ is {\em finite on $M_n$} if there exist $a_1,\ldots,a_k\in M_n$ such that $\{0\}\neq\im (f) \subseteq a_1^\sim\bigcup \cdots \bigcup a_k^\sim$.   Next, let $j$ be a positive integer that divides $n$, choose a primitive $j$-th root of unity $\mu_j$,
denote by ${\bf 1}_r$, $r=\frac{n}{j}$, the identity matrix in $M_r$, and finally, denote by ${\bf w}_j$ the 
 diagonal matrix in $M_n$ having ${\bf 1}_r,\mu_j{\bf 1}_r,\ldots,\mu_j^{j-1}{\bf 1}_r$ on the diagonal.
Our first main result is:

\begin{theorem} [Theorem \ref{koncen}, Corollary \ref{pot}] \label{koncenint}
A polynomial $f$ is finite on
 $M_n$ if and only if there exists a positive integer $j$ dividing $n$  such that $f^j$ is central on $M_n$ and  $f^i$ is not central for  $1\le i < j$. In this case $\im (f) \subseteq  {\bf w}_j^\sim\bigcup n_2^\sim\bigcup  \cdots \bigcup n_k^\sim$, where $n_i$ are nilpotent matrices. Moreover, $\im(f^{j+1})\subseteq {\bf w}_j^\sim$.
 \end{theorem}

It is worth mentioning  that the existence of polynomials whose certain powers
are central is an interesting question that has been studied by several authors (see, e.g., \cite{Alb, AS, Sal}). Yet not everything is  fully understood.

Recall that a subset $U$ of $F^{n^2}$ ($\cong M_n$) is said to be a standard open set (with respect to the Zariski topology) if there exists  $p\in F[z_1,\ldots,z_{n^2}]$ such that $U=\{(u_1,\ldots,u_{n^2})\in F^{n^2}\,|\,p(u_1,\ldots,u_{n^2}) \ne 0\}$. Our second main result is

\begin{theorem}[Theorem \ref{std}]\label{stdint}
If $U$ is a standard open set in $F^{n^2}$ that is closed under nonzero scalar multiplication and  conjugation by invertible matrices, then $U\cup \{0\}=\im (f)$ for some  polynomial $f$.
\end{theorem}

A simple concrete example of such a set $U$ is $GL_n$. 

 The third topic that we consider is the density of $\im (f)$ (with respect to the Zariski topology of $F^{n^2}\cong M_n$).
Given a  noncommutative polynomial $f=f(x_1,\ldots,x_d)$, we can consider
 $\tr(f)$ as a commutative polynomial in $n^2d$ indeterminates. The density of $\im (f)$ can be characterized as follows.

\begin{proposition}[Proposition \ref{generic}]\label{genericint}
The image of a polynomial $f$ is dense in $M_n$ (resp. $M_n^0$ if $\tr(f)=0$) if and only if the polynomials  $\tr(f),\tr(f^2),\dots,\tr(f^n)$ (resp. $\tr(f^2),\dots,\tr(f^n)$) are algebraically independent.
\end{proposition}

Our original motivation for studying the density was the question by Lvov asking whether the image of a multilinear polynomial is a linear space. This was shown to be true for $n= 2$ by Kanel-Belov, Malev and Rowen \cite{Bel}. In general this problem is, to the best of our knowledge, open. If the answer was positive, then either $\im (f)=M_n$ or $\im (f)=M_n^0$ would hold for every  multilinear polynomial $f$ that is neither an identity nor central (see \cite{BK} or \cite{Bel}).  Establishing the density could be an important intermediate step for proving these equalities. On the other hand, as it will be apparent from the next paragraph, this would be sufficient for some applications. 

Motivated by Lvov's problem and Theorems \ref{koncenint} and \ref{genericint}, we have posed ourselves the following two questions concerning a multilinear polynomial $f$.  It has turned out that versions of the first one had already been discussed before (see \cite{Ler,Ro}).

(Q1) If there exists $k\ge 2$ such that $f^k$ is central for $M_n$, $n\ne 2 $,  is then $f$ central?

(Q2) If there exists $k\ge 2$ such that $\tr(f^k)$ vanishes on $M_n$, $n \ne 2$, is then $f$ an identity?

 (Incidentally, the condition that $\tr(f^k)$ vanishes on $M_n$ is equivalent to the condition that $f^k$ is the sum of commutators and an identity \cite{BK}.) 
Note that an affirmative answer to Lvov's question implies that both (Q1) and (Q2) have affirmative answers. Moreover, to establish the latter it would be enough to know only that $\im (f)\cap M_n^0$ is dense in $M_n^0$. Further, since ${\bf w}_j$ has trace zero, one can easily deduce from the last assertion of  Theorem
\ref{koncenint} that an affirmative answer to (Q2) implies an affirmative answer to (Q1). Unfortunately, we were unable to solve any of  these two questions, so we leave them as open problems.  We have only solved the dimension-free version of (Q2):

\begin{proposition}[Proposition \ref{vsotakom}]
If $f $ is a nonzero multilinear polynomial, then $f^k$,  $k\ge 2$, is not a sum of commutators.
\end{proposition}

 In the final part we prove a result giving a small evidence that the answer to Lvov's question may be affirmative. 

\begin{proposition}[Proposition \ref{Lie}]
If $f$ is a nonzero multilinear Lie polynomial of degree at most $4$, then $\im (f)=M_n^0$.
\end{proposition}

\section{preliminaries}\label{pre}
A polynomial $f(x_1,\dots,x_d)$ in the free associative algebra $F\ax=F\langle x_1, . . . , x_d\rangle$ is called a
polynomial identity of $M_n$ if $f(a_1, . . . , a_d) = 0$ for all $a_1, . . . , a_d \in M_n$;
$f\in F\langle x_1, . . . , x_d\rangle$ is a central polynomial of $M_n$ if 
$f(a_1, . . . , a_d) \in F {\bf1}$ for any $a_1, . . . , a_d\in M_n$ but $f$ is not a polynomial identity of $M_n$. 

By $GM(n)$ we denote the algebra of generic matrices over $F$, which is a domain by Amitsur's theorem \cite[Theorem 3.26]{Row}. $UD(n)$ stands for the generic division ring.
The trace of a matrix  can be expressed as a quotient of two central polynomials and can be therefore viewed as an element of $UD(n)$ (see \cite[Corollary 1.4.13, Exercise 1.4.9]{Row}). Since we will need some properties of this expression we repeat  here the form we need.
\begin{proposition}\label{tr}
There exist a multilinear central polynomial $c_0$ and central polynomials $c_1,\dots,c_n$, such that
\begin{equation*}
\tr(a^i)c_0(x_1,\dots,x_t)=c_i(x_1,\dots,x_t,a)
\end{equation*}
for every $a,x_1,\dots,x_{t}\in M_n$, where $t=2n^2$. 
\end{proposition}
By replacing  $a$ by $f(y_1,\dots,y_d)$ we can therefore determine the traces of evaluations of $f$.

It is well-known that the coefficients of the characteristic polynomial can be expressed through the traces as follows:
\begin{proposition}\label{CH0}
There exist $\alpha_{(j_1,\dots,j_n)}\in \Q$ such that the characteristic polynomial can be written as
$$
x^n+\sum_{j=1}^n\sum_{j_1+\dots+j_n=j}\alpha_{(j_1,\dots,j_n)}\tr(x^{j_1})\cdots \tr(x^{j_n})x^{n-j}.
$$
\end{proposition}

A consequence of the above description of the characteristic polynomial is  a well-known fact that a matrix is nilpotent if and only if the trace of each of its powers is zero.

The scalars from Proposition \ref{CH0} can be deduced from Newton's formulas, but we do not need their explicit form.
 However, note that we have a bijective polynomial map from $F^n$ to $F^n$, whose inverse is also a polynomial map, which maps coefficients of the characteristic polynomial of any matrix $x$ into its ``trace" tuple, $(\tr(x),\dots,\tr(x^n))$.  Let us record an easy lemma for  future reference.

\begin{lemma}\label{simsled}
Let $p$ be a symmetric polynomial in $n$ variables. If $f(x)=p(\lambda_1(x),\dots,\lambda_n(x))$, where $\lambda_1(x),\dots,\lambda_n(x)$ are the eigenvalues of a matrix $x\in M_n$, then $f(x)=q(\tr(x),\dots,\tr(x^n))$ for some polynomial $q$.
\end{lemma}

\begin{proof}
Since $p$ is a symmetric polynomial, it can be expressed as a polynomial in the elementary symmetric polynomials  $e_1,\dots,e_n$  by the fundamental theorem of symmetric polynomials. Thus, 
$$f(x)=\tilde{p}\Big (e_1\big(\lambda_1(x),\dots,\lambda_n(x)\big),\dots,e_n\big(\lambda_1(x),\dots,\lambda_n(x)\big)\Big).$$
 Therefore it suffices to prove that $e_i(\lambda_1(x),\dots,\lambda_n(x))=q(\tr(x),\dots,\tr(x^n))$ for some polynomial $q$. Since $\lambda_1(x),\dots,\lambda_n(x)$ are the eigenvalues of a matrix $x$, they are the zeros of the characteristic polynomial of $x$, hence by Vieta's formulas $e_i(\lambda_1(x),\dots,\lambda_n(x))$ equals the coefficient at $x^{n-i}$ in the characteristic polynomial of $x$. The assertion of the lemma follows for every coefficient can be expressed as a polynomial in the traces of powers of $x$ by Proposition \ref{CH0}.
\end{proof}

\section{finite polynomials}\label{pog3}
In this section we want to find the ``smallest possible" images  of polynomials evaluated on $M_n$. If we want a set $S\subseteq M_n$ to be the image of a polynomial we have to require that it is closed under conjugation by invertible matrices. Hence, one possible criterion for the smallness of the image would be the number of similarity orbits contained in it. Therefore one may be inclined to study polynomials that have just a finite number of similarity orbits in their image. However, the images of many polynomials (for example homogeneous as $F$ is algebraically closed) are closed under scalar multiplication,
therefore we take only the  orbits modulo scalar multiplication (by nonzero scalars)  into account. (In this section the expression ``modulo scalar multiplication" will always mean modulo scalar multiplication by nonzero scalars.) In this way we arrive at the definition of a \textit{finite polynomial}, as given in the introduction. 
Central polynomials of $M_n$ have only one nonzero similarity orbit in their image modulo scalar multiplication and we will see that finite polynomials are in close relation with them.

\begin{lemma}\label{powerc}
If $f^j$ is a central polynomial of $M_n$ for some $j\geq 1$, then $f$ is finite on $M_n$.
\end{lemma}

\begin{proof}
If $b=f(\ul{a})$ for some $\ul{a}\in M_n^d$, then the Jordan form of $b$ is either diagonal or nilpotent. If it is diagonal, then it is a scalar multiple of a  matrix having  $j$-th roots of unity on the diagonal. 
There are only finitely many such matrices modulo scalar multiplication. Also the number of similarity orbits of nilpotent matrices modulo scalar multiplication is finite. Thus, $f$ is finite.
\end{proof}

The goal of this section is to prove the converse of this simple observation.

Let us introduce a family of matrices that plays an important role in the next theorem. For every $j$ dividing $n$ choose a primitive $j$-th root of unity $\mu_j$ and define the matrix
$$
{\bf w}_j=
\begin{pmatrix}
{ \bf 1}_r& & & \\
& \mu_j {\bf 1}_r  &  &  \\
& & \ddots & \\
&  & &\mu_j^{j-1}{\bf 1}_r\\
\end{pmatrix},
$$
where $r=\frac{n}{j}$ and ${\bf 1}_r$ denotes the $r\times r$ identity matrix.

A polynomial $f$ is said to be \textit{$j$-central} on $M_n$ if $f^j$ is a central polynomial, while smaller powers of $f$ are not central. We call a polynomial \textit{power-central} if it is $j$-central for some  $j>1$.

\begin{theorem}\label{koncen}
A polynomial $f$ is  finite on $M_n$ if and only if there exists $j\in \mathbb {N}$ such that $f$ is $j$-central on $M_n$. Moreover, in this case every nonnilpotent matrix in $\im(f)$ is similar to a scalar multiple of ${\bf w}_j$.
\end{theorem}

\begin{proof}
Let $a_1,\dots,a_l$ be the representatives of distinct nonnilpotent similarity orbits of $\im(f)$ on $M_n$ modulo scalar multiplication.
For each $a_i$ we set $j_i=\min\{j\;|\tr(a_i^j)\neq 0\}$ (such $j$ exists since $a_i$ is not nilpotent)
and let $\alpha_{ik}=\frac{\tr(a_i^k)^{j_i}}{\tr(a_i^{j_i})^k}$ for $1\leq k\leq n;$ 
these scalars carry the information about the coefficients of the characteristic polynomial of $a_i$. Note that $\alpha_{ik}=0$ for $1\le k\le j_{i}-1$. 
The trace polynomial
$$
\sum_{k=1}^n(\tr(x^k)^{j_i}-\alpha_{ik}\tr(x^{j_i})^k)x_k
$$
vanishes if we substitute a scalar multiple of $a_i$ for $x$ and arbitrary $b_1,\dots,b_n\in M_n$ for $x_1,\dots,x_n$.
Since every nonnilpotent matrix in $\im(f)$ is similar to a scalar multiple of $a_i$ for some $i$ and the trace of powers of nilpotent matrices is zero, the following identity holds in $UD(n)$ (according to Proposition \ref{tr}, all $\tr(f^k)$ lie in $UD(n)$):
$$
\prod_{i=1}^l\Big( \sum_{k=1}^n(\tr(f^k)^{j_i}-\alpha_{ik}\tr(f^{j_i})^k)x_k\Big)=0.
$$
Since $UD(n)$ is a division ring,
one of the factors in the product equals zero in $UD(n)$. 
Hence there exists $i$ such that $\sum_{k=1}^n(\tr(f^k)^{j_i}-\alpha_{ik}\tr(f^{j_i})^k)x_k=0$ and so $\tr(f^k)^{j_i}-\alpha_{ik}\tr(f^{j_i})^k=0$ in $UD(n)$ for every $1\le k\le n$. For simplicity of notation we write $j, \alpha_k$ instead of $j_i,\alpha_{ik}$, respectively, for $1\leq k\leq n$. We will first consider the case when $\alpha_1\neq 0$, i.e., $j=1$ and $\tr(f)\neq 0$. Then the characteristic polynomial of $f$ can be expressed as
\begin{equation}\label{CH}
f^n+\sum_{j=1}^n\beta_j\tr(f)^jf^{n-j}=0 
\end{equation}
 for some $\beta_1,\dots,\beta_n\in F$ (see Proposition \ref{CH0}). 
Let $\lambda_1,\dots,\lambda_n\in {F}$ be zeros of the polynomial $x^n+\sum_{j=1}^n\beta_j x^{n-j}$. Then we can factorize (\ref{CH}) in $UD(n)$ as
$$
\prod_{k=1}^n(f-\lambda_k\tr(f))=0.
$$
This is an identity in $UD(n)$, hence $f-\lambda_k\tr(f)=0$ for some $1\le k\le n$, implying that $f$ is a central polynomial.
Now we consider the general case. We have $\alpha_j\neq 0$ for some $1\leq j\leq n$ and $\alpha_k=0$ for $1\leq k\leq j-1$. Then $\tr(f^j)\neq 0$ and $f^j$ is also finite, so we can just repeat the first part of the proof for $f^j$ from which it  follows that $f^j$ is a central polynomial. In this case $\tr(f^k)=0$ for all $1\le k <j$, therefore $f^k$ is not central.

So far we have proved that $f$ is $j$-central for some $j\ge 1$ and $\tr(f^k)=0$ for all $1\le k <j$. It remains to prove that  nonnilpotent matrices in $\im(f)$ are similar to a scalar multiple of ${\bf w}_j$.
The values of $f$ on $M_n$ can be nilpotent matrices and matrices for which the Jordan form has (modulo scalar multiplication) just powers of the  primitive $j$-th root $\mu_j$ of unity on the diagonal. For simplicity of notation we write $\mu$ instead of $\mu_j$. Take a nonnilpotent matrix $a\in \im(f)$. 
 We are reduced to proving that  the eigenvalues of $a$ are equal to $\lambda,\lambda\mu,\dots,\lambda\mu^{j-1 }$ for some $0\neq\lambda\in {F}$ (depending on $a$) and all have the same algebraic multiplicity $\frac{n}{j}$.
Recall that $\tr(f^k)=0$ for $k<j$.
Hence, if $k_i$ is the multiplicity of $\mu^i$ in the characteristic polynomial of $a\in \im(f)$, then we have 
$$
\begin{array}{*{3}{c@{\:+\:}}c@{\;=\;}c}
k_0 & k_1\mu & \dots & k_{j-1}\mu^{j-1} & 0\\
k_0 & k_1\mu^2 & \dots & k_{j-1}(\mu^{j-1})^2 & 0\\
\multicolumn{5}{c}{\dotfill}\\
k_0 & k_1 \mu^{j-1} & \dots & k_{j-1}(\mu^{j-1})^{j-1} & 0. 
\end{array}
$$
The above equations can be rewritten as $\sum_{i=1}^{j-1}k_i(\mu^t)^i=-k_0
$, $1\leq t\leq j-1$. Having fixed $k_0$, the system of equations in variables $k_1,\dots,k_{j-1}$ will have a unique solution if and only if the determinant of $((\mu^t)^i)$, $1\leq t,i\leq j-1$, is different from zero. 
Since $\mu^t$, $1\leq t\leq j-1$, are distinct, the Vandermonde  argument shows that it is nonzero indeed. Thus, $k_i=k_0$ for every $1\leq i\leq j-1$ is the unique solution. Hence, every nonnilpotent matrix in $\im(f)$ is similar to a scalar multiple of the matrix ${\bf w}_j$.

The converse follows from Lemma \ref{powerc}.
\end{proof}

\begin{corollary}\label{pot}
If $f$ is $j$-central on $M_n$ for some $j\in \N$, then $\im(f^m)$ for $m\geq j$ consists of scalar multiples of exactly one similarity orbit generated by ${\bf w}_j^m$.
\end{corollary}

\begin{proof}
 Since every nonnilpotent matrix in $\im(f)$ is similar to a scalar multiple of ${\bf w}_j$, its $m$-th power is similar to a scalar multiple of ${\bf w}_j^m$. If $f(\ul{a})^m$, $m\geq j$, is  nilpotent, so is $f(\ul{a})^j$. In this case $f(\ul{a})^j=0$ due to the centrality of $f^j$. Hence, $\im(f^m)$, $m\geq j$, does not contain nonzero nilpotent matrices.
\end{proof}

Power-central polynomials are important in the structure theory of division algebras. The question whether $M_p(\Q)$ has a power-central polynomial for a prime $p$ is equivalent to the long-standing open  question whether division algebras of degree $p$ are cyclic. This is known to be true for $p\leq 3$. An example of $2$-central polynomial on $M_2(K)$ for an arbitrary field $K$ is $[x,y]$, which is also multilinear. 
The truth of Lvov's conjecture would imply that there are no multilinear power-central polynomials on $M_n(K)$ for $n\ge 3$. While it is easy to see that multilinear $j$-central polynomials for $j>2$ do not exist over $\Q$ (see, e.g., \cite{Ler}), the same question over an algebraically closed field $F$ remains open.

\begin{remark}
If $f$ is $j$-central, then $\tr(f^2)=0$ if $j>2$, and $\tr(f^3)=0$ if $j=2$. 
Thus, if for multilinear polynomials $f,g$,  the identity $\tr(f^2)=0$ implies $f=0$  (in UD(n)) and the identity $\tr(g^3)=0$   implies $g=0$ (in UD(n)), 
then it would follow that there do not exist multilinear noncentral power-central polynomials.  (See also Section \ref{tr^2}.)
\end{remark}

\section{standard open sets as images of  polynomials}
We will show that if $U$ is a Zariski open subset of $F^{n^2}$, defined as the nonvanishing set of a polynomial in $F[x_{11},\dots,x_{nn}]$ satisfying some natural conditions, then there exists a polynomial $f$ such that $\im(f)=U\cup \{0\}$.
We will first prove that this is true for the most prominent example of such a  set, $GL_n$. 
We follow the standard notation and 
 denote by $V(p)$ the set of zeros of a polynomial $p$, $V(p)=\{(u_1,\dots,u_k)\in F^k|\; p(u_1,\dots,u_k)=0\}$, and by $D(p)=\{(u_1,\dots,u_k)\in F^k|\; p(u_1,\dots,u_k)\neq 0\}$ the complement of $V(p)$. For a subset $V$ of $F^k$ we define $I(V)$ to be the ideal of all polynomials vanishing on $V$, $I(V)=\{p\in F[z_1,\dots,z_k]\,|\; p(u)=0 \text{ for all } u\in V\}$. In this section we will use some basic facts from algebraic geometry which can be found in any standard textbook.

\begin{proposition}
There exists a noncommutative polynomial $f$ such that $\im(f)=GL_n\cup \{0\}$ on $M_n$.
\end{proposition}

\begin{proof}
As $\det(x)$ is a polynomial in the traces of powers of $x$ it can be expressed as the quotient of two central polynomials due to Proposition \ref{tr}. 
We can write $\det(x)=\frac{c(x_1,\dots,x_t,x)}{c_0(x_1,\dots,x_t)^n}$, where $c,c_0$ are central polynomials, $c_0$ is multilinear and $t=2n^2$. Note that if we choose $a_1,\dots,a_{t}$  such that $c_0(a_1,\dots,a_{t})\neq 0$ 
then $\det(x)=0$ if and only if $c(a_1,\dots,a_t,x)\neq 0$. Define $f=c(x_1,\dots,x_t,x)x$.
As $c_0$ is multilinear, $c$ is  homogeneous in the first variable. Therefore $a\in \im(f)$ forces $F a\subseteq \im(f)$ because $F$ is algebraically closed. Hence, the image of $f$ consists of all invertible matrices and the zero matrix.
\end{proof}

In this section we will consider (commutative) polynomials and polynomial maps on $F^{n^2}$. Since these maps will be often evaluated on $n\times n$ matrices we denote the variables by $x_{11},\dots,x_{nn}$.
Let $X$ denote the matrix corresponding to the $n^2$-tuple $(x_{11},\dots,x_{nn})$. By a slight abuse of notation we will sometimes regard a polynomial map $p:F^{n^2}\to F^k$ as a map from $M_n$ to $F^k$. For example, 
$p(x_{11},\dots,x_{nn})=x_{11}+x_{22}+\dots+x_{nn}$ can be seen as a map from $M_n$ to $F$, assigning to every matrix in $M_n$ its trace. In this case we write $p(x_{11},\dots,x_{nn})=\tr(X)$ or even $p(X)=\tr(X)$.
We say that a polynomial map $p$ from $F^{n^2}$ to $F^{n^2}$ is a \textsl{trace polynomial} if $p(x_{11},\dots,x_{nn})=P(X,\tr(X){\bf 1},\dots,\tr(X^n){\bf 1})$ for some polynomial $P(z_0,\dots,z_n)$ with zero constant term.
A polynomial $p:F^{n^2}\to F$ is a \textsl {pure trace polynomial} if $p(x_{11},\dots,x_{nn})=P(\tr(X),\dots,\tr(X^n))$.  (In the previous example we have $P(z_0,\dots,z_n)=z_1$.)

Recall that a polynomial $p:F^{n^2}\to F$ is called a matrix invariant if $p(X)=p(S X S^{-1})$ for every $S\in GL_n$, 
where $p(S X S^{-1})$ denotes the map  that first conjugates the matrix $X$ corresponding to the $n^2$-tuple $(x_{11},\dots,x_{nn})$ with $S$ and then applies $p$ on the $n^2$-tuple corresponding to the matrix $SXS^{-1}$. 
Matrix invariants are exactly the pure trace polynomials  \cite[Theorem 1.5.7]{Spring}. We will use this correspondence without further reference.

\begin{lemma}\label{nicsled}
If $V$ is the zero set in $F^{n^2}$ of trace polynomials $p_1,\dots,p_l$ and if $V$ is closed under scalar multiplication, then there exists a noncommutative homogeneous polynomial $f$ such that $\im (f)=V^\mathsf{c}\cup\{0\}$.
\end{lemma}

\begin{proof}
Let $p_i(x_{11},\dots,x_{nn})=P_i(X,\tr(X){\bf 1},\dots,\tr(X^n){\bf 1})$  for a polynomial $P_i(z_0,z_1,\dots,z_n)$, $1\le i\le l$. 
Let us write $\tr(X^i)=\frac{c_i(X_1,\dots,X_t,X)}{c_0(X_1,\dots,X_t)}$ where $c_0,c_i$, $1\le i\le n$, are  polynomials from Proposition \ref{tr}. 
We replace $P_i(X,\tr(X){\bf 1},\dots,\tr(X^n){\bf 1})$, $1\leq i\leq l$, with $Q_i(X,Y_i)=\tr(P_i(X,\tr(X){\bf 1},\dots,\tr(X^n){\bf 1})Y_i)$, $1\le i\le l$, which map to $F$.
Let $r_i-1$ be the degree of the polynomial $P_i(z_0,z_1,\dots,z_n)$ treated as a polynomial in the last $n$ variables, $z_1,\dots,z_n$. 
Then $c_0(X_1,\dots,X_t)^{r_i}Q_i(X,Y_i)$ is a central polynomial. 

We denote $\ul{X}_i=(X_{i1},\dots,X_{it})$, $1\le i\le l$, $\ul{Y}=(Y_1,\dots,Y_l)$. 
Then 
$c(\ul{X}_1,\dots,\ul{X}_l,X,\ul{Y})=\sum_{i=1}^l c_0(\ul{X}_i)^{r_i}Q_i(X,Y_i)$ is a sum of central polynomials and therefore  
  a central polynomial. If $A\in V$ we have $c(\ul{A}_1,\dots,\ul{A}_l,A,\ul{B})=0$ for any choice of matrices $A_{ij},B_i$. On the other hand, suppose that $A\not\in V$, hence $p_i(a_{11},\dots,a_{nn})\neq 0$ for some $i$. Consequently, there exists $B\in M_n$ such that $Q_i(A,B)\neq 0$. If we choose $A_{i1},\dots,A_{it}$ such that $c_0(\ul{A}_i)\neq 0$ and write $\ul{B}$ for the $l$-tuple that has $B$ on the $i$-th place and zero elsewhere, then 
$c(0,\dots,0,\ul{A}_i,0,\dots,0,A,\ul{B})=\mu {\bf 1}$ for some $0\neq \mu \in F$. 
Let 
$$f(\ul{X}_1,\dots,\ul{X}_l,X,\ul{Y})=c(\ul{X}_1,\dots,\ul{X}_l,X,\ul{Y})X.$$
By construction, $\mu a\in \im(f)$ and since $Q_i(X,Y_i)$ is linear in $Y_i$, all scalar multiples of $A$ belong to the image of $f$ (indeed, $\lambda\mu A=c_0(A_1,\dots,A_t)^{r_i}Q_i(A,\lambda B)A=f(0,\dots,0,\ul{A}_i,0,\dots,0,A,\lambda\ul{B})$ for every $\lambda \in F$).
Hence the image of $f$ equals $V^c\cup\{0\}$.

Since $V$ is closed under scalar multiplication we can assume that $p_i$, $1\le i\le l$, are homogeneous, since otherwise we can replace them by their homogeneous components. These also belong to $I(V)$, which can be easily seen by the Vandermonde argument. The homogeneous components of $p_i$ are also trace polynomials, which follows by comparing  both sides of the equality $p_i(\lambda x_{11},\dots,\lambda x_{nn})=P_i(\lambda X,\tr(\lambda X),\dots,\tr((\lambda X)^n))$.
Hence, we can assume that $P_i(X,\tr(X),\dots,\tr(X^n))$ are homogeneous polynomials of degree $d_i$. 
We denote $d=\max\{d_i+r_it,\;1\le i\le l\}$. If we replace $Q_i(X,Y_i)$ in the above construction by $Q_i(X,Y_i^{d-d_i-r_it+1})$
then $f$ becomes a homogeneous polynomial of degree $d+1$. 
 Noting that $F$ being algebraically closed guarantees that $\im(f)$ is closed under scalar multiplication it is easy to verify that the above proof remains valid  with  polynomials $Q_i(X,Y_i^{d-d_i-r_it+1})$ replacing polynomials $Q_i(X,Y_i)$.
\end{proof}
We illustrate this result with some examples of sets that can be realized as  images of noncommutative polynomials.
\begin{example}
 
(a) The union of matrices that are not nilpotent of the nilindex less or equal to $k$ and the zero matrix is the image of a noncommutative polynomial.
 The matrices whose $k$-th power equals zero are closed under conjugation by $GL_n$ and under scalar multiplication, and they are the zero set of the (trace) polynomial $X^k$. Hence, we can apply Lemma \ref{nicsled}.

(b) Matrices with at most $k$ distinct eigenvalues, $0\le k\le n-1$, are also the zero set of trace polynomials. Define polynomials $p_0(X)=X$, $q_l(z_1,\dots,z_{l+1})=\prod_{1\le i<j\le l+1}(z_i-z_j)^2$ and 
$$p_l(X)=\sum_{1\le i_1<\dots< i_{l+1}\le n} q_l(\lambda_{i_1}(X),\dots,\lambda_{i_{l+1}}(X)),\qquad 1\le l\le n-1,$$
 where $\lambda_i(X)$, $1\le i\le n$, are the eigenvalues of a matrix $X$. Note that the polynomials on the right-hand side of the above definition of $p_l$, $1\le l\le n-1$, are symmetric polynomials in the eigenvalues of the matrix $X$, and thus pure trace polynomials by Lemma \ref{simsled}.
The polynomials $p_l$, $k\le l\le n-1$, define the desired variety. Indeed, $p_{n-1}(X)$ is the discriminant of $X$ and a matrix $A$ is a zero of $p_{n-1}$ if and only if $A$ has at most $n-1$ distinct eigenvalues.
 Then we can proceed by reverse induction to show that the common zeros of $p_{n-1},\dots,p_k$ are  the matrices that have at most $k$ distinct eigenvalues supposing that the common zeros of $p_{n-1},\dots,p_{k+1}$ are the matrices that have at most $k+1$ distinct eigenvalues.
 If $A$ is a zero of $p_{n-1},\dots,p_{k+1}$, i.e. $A$ has at most $k+1$ distinct eigenvalues by the induction hypothesis, then $p_k(A)$ is a scalar multiple of $q_k(\lambda_{1},\dots,\lambda_{{k+1}})\neq 0$ where $\lambda_1,\dots,\lambda_{k+1}$ are possible distinct eigenvalues of $A$.  Therefore, $A$ is a zero of $p_k$ if the evaluation of $q_k$ in this $k+1$-tuple is equal to zero, i.e. if $A$ has at most $k$ distinct eigenvalues. By Lemma \ref{nicsled}, the matrices with at least $k$ distinct eigenvalues together with the zero matrix form the image of a noncommutative polynomial for every $1\le k\le n$.

(c) Define trace polynomials $t_i(X)=\tr(X^i)X-\tr(X)X^i$ for $2\leq i\leq n$. Let a matrix $A$ be a zero of $t_2,\dots,t_n$. Since $A$ is a zero of $t_2$, $A$ is a scalar multiple of an idempotent or $\tr(A)=0$. In the second case, $\tr(A^i)=0$, $1\le i\le n$, since $A$ is a zero of $t_i$, $2\le i\le n$. Thus, 
the variety defined by $t_i$, $2\le i\le n$, contains precisely the scalar multiples of idempotents and nilpotent matrices (only these have the trace of all powers equal to zero). Consequently, the complement of this variety, matrices that are not scalar multiples of an  idempotent and not nilpotent, together with the zero matrix equals the image of a noncommutative polynomial.
\end{example}

We will give two proofs of the following theorem. The first  one might  lead to  possible generalizations, while we find the second one, based on the idea suggested to us  by  Klemen \v Sivic, is quite interesting. 
We first introduce some notation and prove a lemma that will play a role also in the subsequent section.

Let $\phi:M_n\to F^n$ be the  map that assigns to every matrix the coefficients of
its characteristic polynomial. 
 More precisely, if  $x^n+\alpha_1x^{n-1}+\cdots+\alpha_n$ is the characteristic
polynomial of a matrix $a$, then $\phi(a)=(\alpha_1,\dots,\alpha_n)$. 
Note that $\phi$ is a  surjective polynomial map.   

\begin{lemma}\label{fi}
If $Z$ is a proper closed subset of $M_n$ that is closed under conjugation by $GL_n$, then $\phi(Z)$ is contained in a proper closed subset of $F^n$.
\end{lemma}

\begin{proof}
Since the closure of similarity orbits of the set $D$ of all diagonal matrices equals $M_n$, $Z\cap D$ is also a proper closed subset of $D\cong F^n$. Hence $\dim(Z\cap D)<n$. Therefore $\dim(\overline{\phi(Z\cap D)})<n$, which implies that $\overline{\phi(Z\cap D)}$ is a proper closed set of $F^n$. 
Denote by $\tilde{D}$ the set of all diagonalizable matrices. 
As $Z$ is closed under conjugation by $GL_n$, $\phi(Z\cap \tilde{D})=\phi(Z\cap D)$.
Decompose $\phi(Z)=\phi(Z\cap \tilde{D})\cup\phi(Z\cap \tilde{D}^\mathsf{c})$ and notice that $\phi(Z\cap \tilde{D}^c)$ is a subset of the proper closed subset of the variety defined by the discriminant, $V(\disc)$. Hence the closure of $\phi(Z)$ is a proper closed subset of $F^n$.
\end{proof}

\begin{theorem}\label{hiper}
Let $p$ be a commutative polynomial in $n^2$ variables. If $V(p)\subset F^{n^2}$ is closed under conjugation by invertible matrices then $p$ is a pure trace polynomial.
\end{theorem}

\begin{proof}[First proof]
By Lemma \ref{fi}, $\phi(V(p))$ is contained in a proper closed subset of $F^n$.
It thus belongs to $V(f)$ for some polynomial $f$. 
Define $\tilde{f}(X)=f(\alpha_1(X),\dots,\alpha_n(X))$ where $\alpha_1(X),\dots,\alpha_n(X)$ are the coefficients of the characteristic polynomial of a matrix $X$. 
As we have a bijective polynomial correspondence between the ``trace" tuple of a matrix $X$,  $(\tr(X),\dots,\tr(X^n))$, and its ``characteristic" coefficients, $(\alpha_1(X),\dots,\alpha_n(X))$, (see Section \ref{pre}), $\tilde{f}$ is a pure trace polynomial. 
We have  $V(p)\subset V(\tilde{f})$ and  by Hilbert's Nullstellensatz $\tilde{f}^n=pq$ for some $n\in \N$ and some polynomial $q$. 
Since $\tilde{f}$ is a pure trace polynomial we have $\tilde{f}(SXS^{-1})=\tilde{f}(X)$ for every $S\in GL_n$, $X\in M_n$, and, in consequence,  $p(SXS^{-1})q(SXS^{-1})=p(X)q(X)$. 
If $S=(s_{ij})$, then $S^{-1}=\frac{1}{\det(S)}S'$, where $S'$ is a matrix which elements are polynomial functions in $s_{ij}$, $1\le i,j\le n$. 
Thus, we can choose $k,l\in \N$ such that $\overline{p}(S,X)=\det(S)^k p(SXS^{-1})$, $\overline{q}(S,X)=\det(S)^l q(SXS^{-1})$ are polynomials. 
As $F[x_{11},\dots,x_{nn},s_{11},\dots,s_{nn}]$ is a unique factorization domain 
we conclude from $\overline{p}(S,X)\overline{q}(S,X)=\det(S)^{k+l}p(X)q(X)$ that $\overline{p}(S,X)=\det(S)^mp_1(X)$ for some $m\in \Z$ and some polynomial $p_1$,  and hence $p(SXS^{-1})=\det(S)^{m-k} p_1(X)$. 
 Setting $S=1$ yields $p_1=p$, and, in consequence, $p(S)=\det(S)^{m-k}p(S)$ for every $S\in GL_n$, which implies $m=k$. (Indeed, $\det(S)^{m-k}$ has to be equal to $1$ on the open set $D(p)\cap GL_n$, and therefore on the whole $M_n$.)
Hence,  $p$ is a matrix invariant and according to the characterization of matrix invariants a pure trace polynomial.
\end{proof}

\begin{proof}[Second proof]
Firstly, we can assume that $p$ is irreducible. To see this we only need to observe that all irreducible components $V_i$ of $V(p)=\bigcup V_i$ are closed under conjugation by invertible matrices.  Take $X\in V_i$, then the variety $V_X=\overline{\{SXS^{-1},\;S\in GL_n\}}$ is rationally parametrized, and therefore irreducible (see, e.g., \cite[Proposition 4.5.6]{CLO}). Hence, we have $V_X\subseteq V_i$ for every $X\in V_i$, so $V_i$ is closed under conjugation by invertible matrices. In the rest of the proof we therefore assume $p$ to be irreducible. 

We fix an invertible matrix $S$ and define a polynomial $p_S(x_{11},\dots,x_{nn})=p(SXS^{-1})$, which means that we first conjugate the matrix $X$ corresponding to the $n^2$-tuple $(x_{11},\dots,x_{nn})$ with $S$ and then apply $p$ on the $n^2$-tuple corresponding to the matrix $SXS^{-1}$. According to the assumption of the theorem, $p$ and $p_S$ have equal zeros. Hence, $V(p)=V(p_S)$. As $p$ and hence also $p_S$ are irreducible, we have $p_S=\alpha_S p$ for some scalar $\alpha_S\in F$ by Hilbert's Nullstellensatz. We shall have established the lemma if we prove that $\alpha_S=1$ for every $S\in GL_n$.
Indeed, then we can use the characterization of matrix invariants. 
We have $p(SXS^{-1})=\alpha_S p(X)$ for every $S\in GL_n$, $X\in M_n$. In particular, $p(S)=\alpha_S p(S)$, which implies $\alpha_S=1$
for every $S\in U=GL_n\cap D(p)$. Then for every $X\in M_n$ the polynomials $p(SX)$ and $p(XS)$ in $n^2$ variables $s_{11},\dots,s_{nn}$ equal on $U$. Since $U$ is a dense subset of $F^{n^2}$, they are equal. Thus $\alpha_S=1$ for every $S\in GL_n$.
\end{proof}

The next corollary rephrases the last statement in the language of  invariant theory.
\begin{corollary}
If for a polynomial $p:F^{n^2}\to F$  and for every $X\in M_n,S\in GL_n$ we have $p(SXS^{-1})=0$ if and only if  $p(X)=0$, then $p$ is a matrix invariant.
\end{corollary}

Having established Lemma \ref{nicsled} and Theorem $\ref{hiper}$, we can now state the main result of this section.

\begin{theorem}\label{std}
Let $U=D(p)$ be a standard open set in $F^{n^2}$ closed under conjugation by $GL_n$ and nonzero scalar multiplication. There exists a noncommutative homogeneous polynomial $f$ such that $\im(f)=U\cup \{0\}$.
\end{theorem}

To generalize this theorem to arbitrary open subsets of $F^{n^2}$ that are closed under conjugation by $GL_n$ and scalar multiplication with the similar approach (employing Lemma \ref{nicsled}), one would need to prove that every variety that is closed under conjugation by $GL_n$ and under scalar multiplication can be determined by  trace polynomials. (Those trace polynomials may include some extra variables. It is easy to adjust the proof of Lemma \ref{nicsled} to that slightly more general context. See Example \ref{min2} below.) However, we do not know whether this is true or not.

\begin{example}\label{min2}
Let $V$ be the set of all matrices having minimal polynomial of degree at most 2. This is a closed set since each of its element is a zero of the Capelli polynomial $C_5(1,X,X^2,Y,Z)$ for arbitrary $Y,Z\in M_n$, and due to \cite[Theorem 1.4.34]{Row} for $X\not\in V$ there exist $Y,Z\in M_n$ such that $C_5(1,X,X^2,Y,Z)\neq 0$. Hence $c_0(X_1,\dots,X_t)\tr(C_5(1,X,X^2,Y,Z)Y_1)X$ has in its image exactly the zero matrix and matrices whose minimal polynomial has degree at least 3.
\end{example}

\section{density}
Each noncommutative polynomial $f$ in $d$ variables gives rise to  a function 
 $f:M_n^d\to M_n$. In this section we will consider this function as a polynomial
map in $n^2d$ variables. 
 We will be concerned with some  topological aspects of its image on $M_n$. We
discuss the sufficient conditions for establishing the ``dense counterpart" of
Lvov's conjecture. By this we mean the  question whether the image of a
multilinear polynomial $f$ on $M_n$ is dense in $M_n$ or in $M_n^0$, assuming that
$f$ is neither a polynomial identity nor a central polynomial of $M_n$.

Recall that the map $\phi:M_n\to F^n$,  introduced in the previous section, assigns to every matrix the coefficients of
its  characteristic polynomial.
The restriction of  $\phi$ to $M_n^0$ will be denoted by $\phi_0$. Identifying
$\{0\}\times F^{n-1}$ with  $F^{n-1}$, we may and we shall  consider $\phi_0$ as a
map into  $F^{n-1}$.

By saying that $\im(f)$ is dense in $F^{n^2-1}$ we mean that the image of $f$ is
dense in $M_n^0$, an
($n^2-1$)-dimensional  space over $F$,  with the inherited topology from $F^{n^2}$.

\begin{lemma}\label{gost}
Let $f$ be a noncommutative polynomial. Then $\im(f)$ is dense in $F^{n^2}$ (resp. $F^{n^2-1}$ if $\tr(f)=0$) if and only if  $\im(\phi(f))$ (resp. $\im(\phi_0(f))$) is dense in $F^n$ (resp. $F^{n-1}$).
\end{lemma}

\begin{proof}
Assume that $\im(\phi(f))$ is dense in $F^n$. 
Denote by $Z$  the Zariski closure of $\im(f)$. As $\im(f)$ is closed under conjugation by $GL_n$ so is $Z$, thus we can apply Lemma  \ref{fi} to derive that  $Z= F^{n^2}.$
Conversely, if $\im(f)$ is dense in $F^{n^2}$ then   $\im(\phi(f))$ is dense in $F^n$ since $\phi$ is a surjective continuous map.
The respective part can be handled in much the same way, the only difference being the analysis of respective maps  within the framework of $M_n^0$.
\end{proof}

Let $f$ be a noncommutative polynomial depending on $d$ variables.
In the following corollary we regard $\tr(f)$ as a commutative polynomial in $n^2d$ commutative variables.

\begin{proposition}\label{generic}
The image of a polynomial $f$ is dense in $F^{n^2}$ (resp. $F^{n^2-1}$ if $\tr(f)=0$) if and only if $\tr(f),\dots,\tr(f^n)$ (resp. $\tr(f^2),\dots,\tr(f^n)$) are algebraically independent. 
\end{proposition}

\begin{proof}
We have a bijective polynomial map 
from $F^n$ to $F^n$ (whose inverse is also a polynomial map), which maps the coefficients of the characteristic polynomial of an arbitrary matrix $a\in M_n$ to its ``trace" tuple, $(\tr(a),\dots,\tr(a^n))$ (see Section \ref{pre}). Hence $\tr(f),\dots,\tr(f^n)$ are algebraically independent if and only if the coefficients of the characteristic polynomial of $f$ 
 are algebraically independent.

Assume that the coefficients of the characteristic polynomial of $f$ are algebraically dependent. 
Then the image of $\phi(f)$ is contained in a proper algebraic subvariety in $F^n$, which is in particular not dense in $F^n$, therefore $\im(f)$ cannot be dense in $F^{n^2}$. To prove the converse assume that the coefficients of the characteristic polynomial of $f$ are algebraically independent. 
Then the closure of $\im(\phi(f))$ cannot be a proper subvariety and is thus dense in $F^n$. 
We can now apply Lemma \ref{gost} to conclude the proof of the first part. 

The respective part of Lemma \ref{gost} yields in the same manner as above the respective part of this corollary.
\end{proof}

Let $X$ be an irreducible algebraic variety. Recall that the closure of the image of a  polynomial map $p:X\to F^k$ is an irreducible algebraic variety. 
Thus, if $p(X)\cap (F^{k-1}\times \{0\})$ is dense in $F^{k-1}\times \{0\}$ and $p(X)\not\subseteq F^{k-1}\times \{0\}$
 then the (Zariski) closure $Z$ of $p(X)$ equals $F^k$. (Suppose on contrary that $Z=V(p_1,\dots,p_l)\neq F^k$. 
We can assume that $p_1(z_1,\dots,z_k)\neq \alpha z_k^r$ for $r\in \N$, $\alpha\in F$. Write $p_1(z_1,\dots,z_k)=\sum_{i=0}^m q_{i}(z_1,\dots,z_{k-1})z_k^i$ and note that $q_{0}$ equals zero 
since $Z\cap (F^{k-1}\times\{0\})=F^{k-1}\times\{0\}$. 
Thus, there exists the maximal  $r\ge 1$ such that we can write $p_1(z_1,\dots,z_k)=q(z_1,\dots,z_k)z_k^{r}$ for some nonconstant polynomial $q$. 
Hence, $V(p_1)=V(q)\cup V(z_k)$ and, by assumptions and choice of $r$, $Z\neq V(q)\cap Z \neq \emptyset$, 
$Z\neq V(z_k)\cap Z \neq \emptyset$. We derived a contradiction, $Z=Z\cap V(p_1)=(Z\cap V(q))\cup (Z\cap V(z_k))$.)

In the next lemma we will see how the image of a polynomial $f$ evaluated on $M_{n-1}$ impacts $\im(f)$ on $M_n$. In order to distinguish between these  images we write $\im_k(f)$ for $\im(f)$ evaluated on $M_k$. We identify  $M_{n-1}$ with 
$\left(\begin{array}{ccc}
M_{n-1}& 0\\
0 & 0
\end{array}\right)
$
 inside $M_n$.

\begin{lemma}\label{n-1,n}
If $\im_{n-1}(f)\cap M_{n-1}^0$ is dense in $M_{n-1}^0$ then $\im_n(f)$ is dense is $M_n^0$. If, additionally, $\im_n(f)\not\subseteq M_n^0$ then $\im_n(f)$ is dense in $M_n$.
\end{lemma}

\begin{proof}
Assume that $\im_{n-1}(f)\cap M_{n-1}^0$ is dense in $M_{n-1}^0$. Therefore $\im_n(\phi(f))\cap (\{0\}\times F^{n-2}\times \{0\}) $ is dense in $\{0\}\times F^{n-2}\times \{0\}$. (The last component of the polynomial map $\phi(f)$ is $\det(f)$.) According to the discussion preceding the lemma, $\im(\phi(f))\cap(\{0\}\times F^{n-1})$ is dense in $\{0\}\times F^{n-1}$ if it contains an invertible matrix.
The later was observed in \cite[Theorem 2.4]{LeZhou}. 
Thus, $\im_n(f)\cap M_n^0$ is dense in $M_n^0\cong F^{n^2-1}$ by Lemma \ref{gost}. 
If, additionally, there exists a matrix in the image of $f$ with nonzero trace, $\im_n(f)$ is dense in $M_n$ by the above discussion identifying $M_n^0$ with $F^{n^2-1}$.
\end{proof}

\begin{corollary}
If a multilinear polynomial $f$ is neither a polynomial identity nor a central polynomial of $M_2$, then $\im(f)$ is dense in $M_n$ for every $n\geq 2$.
\end{corollary}

\begin{proof}
Apply \cite[Theorem 2]{Bel} and Lemma \ref{n-1,n}.
\end{proof}

In view of Lemma \ref{n-1,n} it would suffice to verify the density version of Lvov's conjecture for a polynomial $f$ evaluated on $M_n$ for such $n$ that $f$ is a polynomial identity or a central polynomial of $M_{n-1}$ and is not a polynomial identity or a central polynomial of $M_n$. 
The first step in this direction may be to establish the density of the image of the standard polynomials $St_n$.

The following questions arise when trying to establish a connection between  Lvov's conjecture and its dense counterpart. 
Does the density of $\im(f)$ in $M_n$ or in $M_n^0$ for a multilinear polynomial $f$ imply that $\im(f)=M_n$ or $M_n^0$, respectively? Is the image of a multilinear polynomial closed in $F^{n^2}$?

\begin{remark}
The image of a homogeneous polynomial is not necessarily closed in $F^{n^2}$. Lemma \ref{nicsled} provides examples of such homogeneous polynomials.
\end{remark}

\begin{remark}
We were dealing with the Zariski topology, however, if the underlying field $F$ equals $\CC$, the field of complex numbers, all statements remain valid when we replace the Zariski topology with the (more familiar) Euclidean topology. This rests on the result from algebraic geometry (see, e.g.,  \cite[Theorem 10.2]{Mil}) asserting that the image of a polynomial map $g$ contains a Zariski open set of its closure. Thus, if the image of $g:\CC^m\to \CC^k$ is dense in the Zariski topology,
then it contains a dense open subset, which is clearly open and also dense in $\CC^k$ in the Euclidean topology. (Indeed, its complement, which is a set of zeros of some polynomials,  cannot contain an open set.) Consequently, the image of a polynomial map $g:\CC^m\to\CC^k$ is dense in the Zariski topology in $\CC^k$ if and only if it is dense in the Euclidean topology. However, the question whether the image of a multilinear polynomial $f:M_n(\CC)^{d}\to M_n(\CC)$ is closed in the Euclidean topology in $M_n(\CC)$ might be approachable with tools  of complex analysis.
\end{remark}

\section{zero trace squares of polynomials}\label{tr^2}
Let $f$ be a polynomial that is not an identity of $M_n$.
The simplest situation where the conditions of  Corollary \ref{generic} are not fulfilled is when $\tr(f^2 )=0$. Let us first show that this can actually occur. The proof of the next proposition is due to Igor Klep who has kindly allowed us to include it here.

\begin{proposition}
Let $n=2^m \ell$, where $\ell>1$ is odd. Then there exists a multihomogeneous polynomial $f$ which is not a polynomial identity of $M_n$ with $\tr(f^2)=0$ on $M_n$.
\end{proposition}

\begin{proof}
\def\cD{\mathcal D}
\def\cZ{\mathcal Z}
Consider the universal division algebra 
$\cD=UD(n)$.
$\cD$ comes equipped with the reduced trace $\tr:\cD\to\cZ=Z(\cD)$. 

We claim that the quadratic trace form $q: x\mapsto \tr(x^2)$ on $\cD$ is isotropic. 
Since $\ell>1$, there is an odd degree extension $K$ of $\cZ$ such that $\cD\otimes_{\cZ} K = M_l(K\tnz_{\cZ}\cD')$ where $\cD'$ is a division ring  (see, e.g., \cite[Theorem 3.1.40]{Row}).
The natural extension $q_K$ of $q$ to $M_l(K\tnz_{\cZ}\cD')$ is 
obviously isotropic, i.e., there is $A\in M_l(K\tnz_{\cZ}1)$ with $q_K(A)=\tr(A^2)=0$.
Hence by Springer's theorem \cite[Corollary 18.5]{EKM},  $q$ is isotropic as well.
There exists $0\neq y\in\cD$ with $\tr(y^2)=0$.
We have $y=f c^{-1}$ for some $f\in GM(n)$
and $c\in Z(GM(n))$.  Replacing $y$ by $c y$, 
we may assume 
without loss of generality that $y\in GM(n)$. There is $f\in F\ax$ whose image
in $GM(n)$ coincides with $y$.

By the universal property of the reduced trace on $\cD$, $\tr(y^2)=0$ translates into
$\tr (f(\ua)^2)=0$ for all $n$-tuples $\ua$ of $n\times n$ matrices over $F$. By
(multi)homogeneizing we can even achieve that $0\neq f$ is multihomogeneous.
\end{proof}

As explained in the introduction of the paper, one would expect that multilinear polynomials that are not identities cannot satisfy 
$\tr(f^k)=0$ for $k\ge 2$. Unfortunately, we are able to prove this only in the dimension-free setting.
That is, we consider the situation where $f$ satisfies $\tr(f^k)=0$ on $M_n$ for every $n\ge 1$. This is equivalent to the condition that $f^k$ is a sum of commutators \cite[Corollary 4.8]{BK}.

\begin{proposition}\label{vsotakom}
If $f \in F\ax$ is a nonzero multilinear polynomial, then $f^k$, $k\geq 2$, is not a sum of commutators.
\end{proposition}

\begin{proof}
To avoid notational difficulties, we will consider only the case where $k=2$.
The modificatons needed to cover the general case are rather obvious.

If $f^2$ is a sum of commutators then  $\tr(f^2)=0$ in matrix algebras of  arbitrary dimension. 
Let us write $f=f(x_1,\dots,x_d)=\sum f_i x_1 g_i$.
 Since $\tr(f(x+y,x_2,\dots,x_d)^2-f(x,x_2,\dots,x_d)^2-f(y,x_2,\dots,x_d)^2)=0$, we have $\tr(x(\sum_{i,j}g_i f_j y g_j f_i))=0$. This implies that 
$\sum_{i,j}g_i f_j x_1 g_j f_i$ is a polynomial identity for every matrix algebra, so it has to be trivial. Denote by $f^*$ the Razmyslov transform of $f$ according to $x_1$, 
$f^*=\sum g_i x_1 f_i$. We have $f^*(f(x_1,x_2,\dots,x_d),x_2,\dots,x_d)=0$ in the free algebra  $F\ax$, which further yields $f^*=0$. Indeed, suppose $f^*\neq 0$ and choose monomials $m_1,m_2$ with nonzero coefficients in $f$ and $f^*$, respectively, which are minimal due to the first appearance of $x_1$.
Then the coefficient of the monomial  $m_2(m_1,x_2,\dots,x_d)$ in the polynomial $f^*(f(x_1,x_2,\dots,x_d),x_2,\dots,x_d)$ is nonzero, a contradiction. Hence, $f^*$ has to be zero, which leads to the contradiction $f=0$ ($f^*=0$ if and only if $f=0$, see, e.g., \cite[Proposition 12]{For}).  
\end{proof}

\begin{remark}
From the proof we deduce that only the linearity in one variable is needed. However, for general polynomials we were not able to find out whether $f^k$ can be a sum of commutators. In any case, this problem can be just a test for a more general question (see Question \ref{M8}).
\end{remark}

Let $M_\infty$ denote the algebra of all infinite matrices with finitely many nonzero entries. We write $M_\infty ^0$ for the set of elements in $M_\infty$ with zero trace, where the trace is defined as the sum of diagonal entries. 

\begin{question}\label{M8}
Is the image of an arbitrary noncommutative polynomial $f$  on $M_\infty$ a dense subset (in the Zariski topology) of $M_\infty$ or of $M_\infty^0$? Is $\im(f)=M_\infty$ or $\im(f)=M_\infty^0$?
\end{question}

 If $f^k$ is a sum of commutators for some polynomial $f$ and some $k>1$, then $\im(f)$ on $M_\infty$ is not dense in $M_\infty$, hence such a polynomial $f$ would provide a counterexample to the above question.

The question about the density in the sense of the Jacobson density theorem was settled in \cite{C-L}.

\section{Lie polynomials of degree $2,3,4$}
We prove that Lvov's conjecture holds for multilinear Lie polynomials of degree less or equal to $4$.
We use the  right-normed notation, $[x_n,\dots,x_1]$ denotes $[x_n,[x_{n-1},[\dots[x_2,x_1]]\dots]$.

\begin{lemma}\label{baza}
If $f$ is of the form $f(x_1,\dots,x_d)=[x_{i_1},x_{i_2},\dots,x_{i_{k-1}},x_1]$, where $2\leq i_j\leq d$, then $\im (f)=M_n^0$.
\end{lemma}

\begin{proof}
Choose  a diagonal matrix $s$ with distinct diagonal entries $\lambda_i$, $1\leq i\leq n$.
Then $f(x,s,\dots,s)_{ij}=\pm (\lambda_i-\lambda_j)^{k-1}x_{ij}$, where $x=(x_{ij})$. Thus, $\im (f)$ contains all matrices with zero diagonal entries. 
Since $\im(f)$ is closed under conjugation and every matrix with zero trace is similar to a matrix with zero diagonal (see, e.g.,  \cite{Shoda}), we have
$\im(f)=M_n^0$.
\end{proof}

If $f$ is a Lie polynomial of degree 2, $f=\alpha [x_1,x_2]$, $\alpha\neq 0$, it has been known for a long time \cite{Shoda, Albert} that $\im(f)=M_n^0$. We list this as a lemma for the sake of reference.

\begin{lemma}\label{2}
If $f$ is a Lie polynomial of degree 2, then $\im(f)=M_n^0$.
\end{lemma}

\begin{lemma}\label{3}
If $f$ is a multilinear Lie polynomial of degree $3$, then $\im(f)=M_n^0$.
\end{lemma}

\begin{proof}
We can assume that $f(x,y,z)=[z,y,x]+\alpha[y,z,x]$. If we take $x=z$, we have $f(x,y,x)=[x,y,x]=-[x,x,y]$. 
We apply Lemma \ref{baza} to conclude $M_n^0=\im(f(x,y,x))\subseteq \im(f)\subseteq M_n^0$.
\end{proof}

\begin{lemma}\label{4}
If $f$ is a multilinear Lie polynomial of degree $4$, then $\im(f)=M_n^0$.
\end{lemma}

\begin{proof}
It is easy to see that the monomials $[x_i,x_j,x_k,x_1]$, $\{i,j,k\}=\{2,3,4\}$, form a basis of multilinear Lie polynomials of degree 4. We can assume that 
\begin{eqnarray*}
f(x_1,x_2,x_3,x_4)&=&[x_4,x_3,x_2,x_1]+\alpha_1[x_3,x_4,x_2,x_1]+\alpha_2[x_4,x_2,x_3,x_1]+\alpha_3[x_2,x_4,x_3,x_1]+\\
&&\alpha_4[x_3,x_2,x_4,x_1]+\alpha_5[x_2,x_3,x_4,x_1].
\end{eqnarray*}
Consider $f(x,x,x,y)=(\alpha_4+\alpha_5)[x,x,y,x]=-(\alpha_4+\alpha_5)[x,x,x,y]$. Due to Lemma \ref{baza} we can assume $\alpha_4+\alpha_5= 0$.
Similarly, setting $x_1=x_2=x_4$ and $x_1=x_3=x_4$ yields $\alpha_2+\alpha_3=0$ and $1+\alpha_1=0$, respectively.  
Hence we can write 
\begin{eqnarray*}
f(x_1,x_2,x_3,x_4)&=&[x_4,x_3,x_2,x_1]-[x_3,x_4,x_2,x_1]+\alpha_2([x_4,x_2,x_3,x_1]-[x_2,x_4,x_3,x_1])+\\
& &\alpha_4([x_3,x_2,x_4,x_1]-[x_2,x_3,x_4,x_1])\\
                                &=&[[x_4,x_3],x_2,x_1]+\alpha_2[[x_4,x_2],x_3,x_1]+\alpha_4[[x_3,x_2],x_4,x_1].
\end{eqnarray*}
Then $f(x,y,y,z)=(1+\alpha_2)[[z,y],y,x]=(1+\alpha_2)[[z,y],[y,x]]$.
It follows from the proof of \cite[Theorem 1]{Smith}  that any matrix with zero trace except for the rank-one matrix is similar to a commutator of two matrices
with zero diagonal.  Choose a diagonal matrix $s$ with distinct diagonal entries. Take $a\in M_n^0$ with rank at least 2. Since $\im(f)$ is closed under conjugation by $GL_n$, we may assume that $a=[b,c]$ where $b,c\in M_n$ have zero diagonal.
Hence, we can write $b=[b',s],\; c=[c',s]$ for some $b',c'\in M_n$. Thus $(1+\alpha_2)a=f(b',s,s,c')$. If $a\in M_n^0$ has rank one, then $a$ is similar to the matrix unit $e_{12}$ and $(1+\alpha_2)e_{12}=f(e_{21},e_{12},e_{12},-\frac{1}{2}e_{11})$. Hence, $1+\alpha_2=0$ or $\im(f)$ contains all matrices with zero trace. Therefore we can assume $1+\alpha_2=0$. If we set $x_1=x_4$ we get in a similar way that $1-\alpha_2=0$, which leads to a contradiction. Hence, $M_n^0\subseteq \im(f)\subseteq M_n^0$ yields the desired conclusion.
\end{proof}

\begin{proposition}\label{Lie}
If $f$ is a nonzero multilinear Lie polynomial of degree at most $4$, then $\im(f)=M_n^0$.
\end{proposition}

\begin{proof}
Apply Lemmas \ref{2}, \ref{3}, \ref{4}.
\end{proof}

{\bf Acknowledgement.} The author would like to thank to her supervisor Matej Bre\v sar for many fruitful discussions and insightful comments, and to Igor Klep and Klemen \v Sivic for generous sharing of ideas.

\end{document}